\documentclass[a4paper,11pt,twoside]{amsart}

\usepackage[left=2.7cm,right=2.7cm,top=3.5cm,bottom=3cm]{geometry}

\usepackage{amsmath,amssymb,amsthm,amscd,amsfonts}
\usepackage{mathrsfs}
\usepackage{latexsym}
\usepackage[all]{xy}
\usepackage{eucal}
\usepackage{verbatim}
\usepackage[usenames,dvipsnames]{color}

\newcommand{\Q}{\mathbb{Q}}

\newcommand{\Z}{\mathbb{Z}}
\newcommand{\F}{\mathbb{F}} 
\newcommand{\M}{\mathfrak{M}} 
\renewcommand{\H}{\mathcal{H}}
\newcommand{\E}{{\rm E}} 
\newcommand{\D}{\mathcal{D}}
\newcommand{\G}{\mathcal{G}}

\newcommand{\Gal}{{\rm Gal}}
\newcommand{\Hom}{{\rm Hom}}
\newcommand{\Ind}{{\rm Ind}}
\newcommand{\Coind}{{\rm Coind}}
\newcommand{\Ext}{{\rm Ext}}

\newcommand{\cd}{{\rm cd}}
\newcommand{\ri}{\rightarrow}
\newcommand{\sri}{\twoheadrightarrow}
\newcommand{\iri}{\hookrightarrow}
\renewcommand{\k}{\kappa} 
\renewcommand{\l}{\ell}
\renewcommand{\d}{\delta}

\renewcommand{\L}{\Lambda}

\newcommand{\ov}{\overline}

\newcommand{\rank}{{\rm rank}}
\newcommand{\dl}[1]{\lim_{\buildrel \longrightarrow\over{#1}}}
\newcommand{\il}[1]{\lim_{\buildrel \longleftarrow\over{#1}}}

\newtheorem{defin}{Definition}[section]
\newtheorem{prop}[defin]{Proposition}
\newtheorem{lem}[defin]{Lemma}
\newtheorem{rem}[defin]{Remark}
\newtheorem{thm}[defin]{Theorem}

\newtheorem{cor}[defin]{Corollary}
\newtheorem{exe}[defin]{Example}

\input cyracc.def
\font\tencyr=wncyr10
scaled \magstephalf
\def\cyr{\tencyr\cyracc}
\newcommand{\ts}{\mbox{\cyr Sh}}

\title[On Selmer groups of abelian varieties]{On Selmer groups of abelian varieties over $\ell$-adic Lie
extensions of global function fields}

\author{A. Bandini}
\author{M. Valentino}

\begin{document}

\begin{abstract}
Let $F$ be a global function field of characteristic $p>0$ and $A/F$ an abelian
variety. Let $K/F$ be an $\l$-adic Lie extension ($\l\neq p$) unramified 
outside a finite set of primes $S$ and such that $\Gal(K/F)$ has no elements
of order $\l$. We shall prove that, under certain conditions, $Sel_A(K)_\l^\vee$
has no nontrivial pseudo-null submodule. 
\end{abstract}

\maketitle



\section{Introduction}
Let $G$ be a compact $\l$-adic Lie group and $\L(G)$ its associated Iwasawa algebra.
A crucial theme in Iwasawa theory is the study of finitely generated $\L(G)$-modules
and their structure, up to ``pseudo-isomorphism''.
When $G\simeq \Z_\l^d$ for some integer $d\geqslant 1$, the structure theory for finitely
generated $\L(G)$-modules is well known (see, e.g., \cite{B}). For a nonabelian $G$, which is the case we are
interested in, studying this topic is possible thanks to an appropriate definition of the concept of ``pseudo-null''
for modules over $\L(G)$ due to Venjakob (see \cite{Vj}).

Let $F$ be a global function field of trascendence degree one over its constant
field $\F_q$, where $q$ is a power of a fixed prime $p\in \Z$, and $K$ a Galois extension of $F$ unramified
outside a finite set of primes $S$ and such that $G=\Gal(K/F)$
is an infinite $\l$-adic Lie group ($\l\in\Z$ a prime different from $p$).
Let $A/F$ be an abelian variety: in \cite{BV}, we proved that $\mathcal{S}:=Sel_A(K)_\l^\vee$ (the Pontrjagin
dual of the Selmer group of $A$ over $K$) is finitely generated over $\L(G)$ and
here we shall deal with the presence of nontrivial pseudo-null submodules in $\mathcal{S}$.
For the number field setting and $K=F(A[\l^\infty])$, this issue was studied by Ochi and Venjakob
(\cite[Theorem 5.1]{OV}) when $A$ is an elliptic curve, and by Ochi for a general abelian variety
in \cite{O}.

In Sections \ref{SetNot} and \ref{SecVenj} we give a brief description of the objects we will work with and of the main
tools we shall need, adapting some of the techniques of \cite{OV} to our function field setting and to a general
$\l$-adic Lie extension (one of the main difference being the triviality of the image of the local Kummer maps).

In Section \ref{SecMain} we will prove the following

\begin{thm}[Theorem \ref{NoPsNThm}]
Let $G=\Gal(K/F)$ be an $\l$-adic Lie group without elements of order $\l$ and of
positive dimension $d\geqslant 3$. If
$H^2(F_S/K, A[\l^\infty])=0$ and the map $\psi$ (induced by restriction)
\begin{equation*}Sel_A(K)_\l \iri H^1(F_S/K, A[\l^\infty])
\stackrel{\psi}{\longrightarrow} \bigoplus_S \Coind_G^{G_v} H^1(K_w, A)[\l^\infty]
\end{equation*} is surjective, then $Sel_A(K)_\l^\vee$ has no nontrivial pseudo-null submodule.
\end{thm}

\noindent For the case $d=2$ we need more restrictive hypotheses, in particular we have the following

\begin{prop}[Proposition \ref{dimcm2}]
Let $G=\Gal(K/F)$ be an $\l$-adic Lie group without elements of order $\l$ and of
dimension $d\geqslant 2$. If $H^2(F_S/K, A[\l^\infty])=0$ and
$\cd_\l(G_v)=2$ for any $v\in S$,
then $Sel_A(K)_\l^\vee$ has no nontrivial pseudo-null submodule.
\end{prop}

\noindent A few considerations and particular cases for the vanishing of $H^2(F_S/K, A[\l^\infty])$ are included
at the end of Section \ref{SecMain}.\\

\noindent {\sc Acknowledgements} The authors thank F. Bars for valuable discussions and comments on earlier drafts
of this paper.

\section{Setting and notations}\label{SetNot}
Here we fix notations and conventions that will be used through the paper.

\subsection{Fields and extensions} Let $F$ be a global function field of trascendence degree
one over its constant field $\F_F=\F_q$, where $q$ is a power of a fixed
prime $p\in \Z$. We put $\ov{F}$ for an algebraic closure of $F$. \\
For any algebraic extension $L/F$, let $\M_L$ be the set
of places of $L$: for any $v\in \M_L$ we let $L_v\,$ be the completion of $L$ at $v$.
Let $S$ be a finite nonempty subset of $\M_F$ and let $F_S$ be the maximal Galois extension of $F$ unramified
outside $S$ with $G_S(F):=\Gal(F_S/F)$. Put $\mathcal{O}_{L,S}$ as the ring of $S$-integers of $L$
and $\mathcal{O}_S^\times$ as the units of $\mathcal{O}_S=\displaystyle{\bigcup_{L\subset F_S} \mathcal{O}_{L,S}}\,$.
Finally, $C\l_S(L)$ denotes the {\em $S$-ideal class group of $ \mathcal{O}_{L,S}$}: since $S$ is nonempty,
$C\l_S(L)$ is finite.\\
For any place $v\in \M_F$ we choose (and fix) an embedding $\ov{F} \iri \ov{F_v}\,$, in order to get a
restriction map $G_{F_v}:=\Gal(\ov{F_v}/F_v) \iri G_F:=\Gal(\ov{F}/F)\,$.\\

\noindent We will deal with {\em $\l$-adic Lie extensions} $K/F$, i.e.,  Galois
extensions with Galois group an $\l$-adic Lie group with $\l\neq p$. We always assume that our
extensions are unramified outside a finite set $S$ of primes of $\M_F\,$.

\noindent In what follows $\Gal(K/F)$ is an $\l$-adic Lie group {\em without points of order} $\l$, then
it has finite $\l$-cohomological dimension, which is equal to its dimension as an $\l$-adic
Lie group (\cite[Corollaire (1) p. 413]{Se2}).

\subsection{Ext and duals} For any $\l$-adic Lie group $G$ we denote by
\[ \L(G) = \Z_\l[[G]] := \il{U} \Z_\l [G/U] \]
the associated {\em Iwasawa algebra} (the limit is on the open normal subgroups of $G$).
From Lazard's work (see \cite{L}), we know that $\L(G)$ is Noetherian and, if $G$ is pro-$\l$
and has no elements of order $\l$, then $\L(G)$ is an integral domain.\\

\noindent For a $\L(G)$-module $M$ we consider the extension groups
\[ \E^i(M):={\rm Ext}^i_{\L(G)}(M,\L(G)) \]
for any integer $i$ and put $\E^i(M)=0$ for $i<0$ by convention.\\
Since in our applications $G$ comes from a Galois extension, we denote
with $G_v$ the decomposition group of $v\in\M_F$ for some prime $w|v$, $w\in \M_L$, and we use the notation
\[  \E^i_v(M):= {\rm Ext}^i_{\L(G_v)}(M,\L(G_v)) \ .\]

\noindent Let $H$ be a closed subgroup of $G$. For every $\L(H)$-module $N$ we consider the
$\L(G)$-modules
\[ \Coind^H_G(N):={\rm Map}_{\L(H)}(\L(G), N)\quad {\rm and}\quad \Ind_H^G(N):=N \otimes_{\L(H)}\L(G) \
\footnote{We use the notations of \cite{OV}, some texts, e.g. \cite{NSW}, switch the definitions of $\Ind^H_G(N)$ and $\Coind^H_G(N)$.}. \]

\noindent For a $\L(G)$-module $M$, we denote its Pontrjagin dual by $M^{\vee}:=\Hom_{cont}(M,\Q_\l/\Z_\l)$.
In this paper, $M$ will be a (mostly discrete) topological $\Z_\l$-module, so $M^{\vee}$ has a natural
structure of $\Z_\l$-module.

\noindent If $M$ is a discrete $G_S(F)$-module, finitely generated over $\Z$ and with no $p$-torsion,
in duality theorems we shall use also the dual $G_S(F)$-module of $M$, i.e.,
\[ M':=\Hom(M, \mathcal{O}_S^\times)\ (=\Hom(M, \boldsymbol{\mu})\ {\rm if}\ M\ {\rm is\ finite})\ . \]

\subsection{Selmer groups} Let $A$ be an abelian variety of dimension $g$ defined over $F$:
we denote by $A^t$ its dual abelian variety. For any positive integer
$n$ we let $A[n]$ be the scheme of $n$-torsion points and, for any prime $\l$, we put
$A[\l^\infty]:={\displaystyle \lim_{\rightarrow}}\, A[\l^n]\,$.\\

\noindent The {\em local Kummer maps} (for any $w\in\M_L\,$)
\[  \k_w : A(L_w) \otimes \Q_\l/\Z_\l \iri \dl{n} H^1(L_w,A[\l^n]):=
H^1(L_w,A[\l^\infty]) \]
(arising from the cohomology of the exact sequence $A[\l^n]\iri A{\buildrel \l^n \over\sri} A$)
enable one to define the {\em $\l$-part of the Selmer group} of $A$ over $L$ as
\[ Sel_A(L)_\l=Ker \left\{ H^1(L, A[\l^{\infty}]) \to \prod_{w\in\M_L}
H^1(L_w, A[\l^{\infty}])/Im\,\k_w \right\} \]
(where the map is the product of the natural restrictions between cohomology groups).\\
For infinite extensions $\mathcal{L}/F$ the Selmer group $Sel_A(\mathcal{L})_\l$
is defined, as usual, via direct limits.\\

\noindent Since $\l\neq p$, the $Im\,\k_w$ are trivial and, assuming that $S$ contains also
all primes of bad reduction for $A$, we have the following equivalent
\begin{defin}\label{DefSel} {\em The} $\l$-part of the Selmer group
{\em of $A$ over $L$ is}
\[Sel_A(L)_\l = Ker \left\{ H^1(F_S/L, A[\l^\infty](F_S)) \to
\bigoplus_S \Coind^{G_v}_G H^1(L_w, A[\l^\infty]) \right\} \ . \]
\end{defin}

\noindent Letting $L$ vary through subextensions of $K/F$, the groups $Sel_A(L)_\l$ admit natural
actions by $\mathbb{Z}_\l$ (because of $A[\l^\infty]\,$) and by $G=\Gal(K/F)$. Hence they are
modules over the Iwasawa algebra $\L(G)$.\\

\section{Homotopy theory and pseudo-nullity}\label{SecVenj}
We briefly recall the basic definitions for pseudo-null modules over a non-commutative Iwasawa algebra:
a comprehensive reference is \cite{Vj}.

\subsection{Pseudo-null $\L(G)$-modules} Let $G$ be an $\l$-adic Lie group without $\l$-torsion, 
then $\L(G)$ is an Auslander regular ring of finite global dimension $\mathfrak{d}=\cd_\l(G)+1$
(\cite[Theorem 3.26]{Vj}, $\cd_\l$ denotes the $\l$-cohomological dimension).\\
For any finitely generated $\L(G)$-module $M$, there is a canonical filtration
\[ T_0(M) \subseteq T_1(M) \subseteq \cdots \subseteq T_{\mathfrak{d}-1}(M) \subseteq T_{\mathfrak{d}}(M)=M \ . \]
\begin{defin}\label{DefPsN}
{\em We say that a $\L(G)$-module $M$ is} pseudo-null {\em if \[\d(M):={\rm min}\{ i\,|\,T_i(M)=M\}\leqslant \mathfrak{d}-2 \ .\]}
\end{defin}

\noindent The quantity $\d(M)$, called the {\em $\d$-dimension} of the $\L(G)$-module
$M$, is used along with the {\em grade} of $M$, that is
\[ j(M):={\rm min}\{i\,|\, \E^i(M)\neq 0 \} \ .\]
As $j(M)+\d(M)=\mathfrak{d}$ (\cite[Proposition 3.5 (ii)]{Vj}) we have that $M$ is a pseudo-null module
if and only if $\E^0(M)=\E^1(M)=0$.\\
Since $\d(T_i(M))\leqslant i$ and every $T_i(M)$ is the maximal submodule of $M$ with
$\d$-dimension less or equal to $i$ (\cite[Proposition 3.5 (vi) (a)]{Vj}), only $T_0(M)\,,\ldots,\,T_{\mathfrak{d}-2}(M)$
can be pseudo-null. If $T_0(M)\,=\cdots=\,T_{\mathfrak{d}-2}(M)=0$, $M$ does not have any nonzero
pseudo-null submodule. This is the case when $\E^i\E^i(M)=0\ \forall\, i\geqslant 2$
(\cite[Proposition 3.5 (i) (c)]{Vj}).

\subsection{The powerful diagram and its consequences} In \cite[Lemma 4.5]{OV} Ochi and Venjakob ge\-ne\-ra\-li\-zed
a result of Jannsen (see \cite{J}) which is very powerful in applications (they call it ``powerful diagram'').
We provide here the statements we shall need later: for the missing details of the proofs the reader can
consult \cite[Chapter V, Section 5]{NSW} and/or \cite[Section 4]{OV}\footnote{Those results hold in our setting
as well because we work with the $\L(G)$-module $A[\l^\infty]$, with $\l\neq p$.}.\\
Replacing, if necessary, $F$ by a finite extension we can (and will) assume that $K$
is contained in the maximal pro-$\l$ extension of $F_\infty:=F(A[\l^\infty])$ unramified
outside $S$. Then we have the following
\[ \begin{xy}
(40,15)*+{F_\infty}="v2";
(30,15)*+{K}="v3";
(30,30)*+{\Omega}="v4";
(30,40)*+{F_S}="v5";
(30,0)*+{F}="v1";
{\ar@{-}"v1";"v2"};
{\ar@{-}^{G}"v1";"v3"};
{\ar@{-}^{\mathcal{H}}"v3";"v4"};
{\ar@{-}"v2";"v4"};
{\ar@{-}"v4";"v5"};
{\ar@{-}@/_{2pc}/_{\mathcal{G}}"v4";"v1"};
\end{xy} \]
where $\Omega$ is the maximal pro-$\l$ extension of $F_\infty$ contained in $F_S$.
We put $\mathcal{G}=\Gal(\Omega/F)$, $\mathcal{H}=\Gal(\Omega/K)$ and $G=\Gal(K/F)$.
The extension $F_\infty/F$ will be called the {\em trivializing extension}.\\

\noindent Tensoring the natural exact sequence $I(\mathcal{G}) \iri \L(\mathcal{G})  \sri \Z_\l $
with $A[\l^\infty]^\vee\simeq \Z_\l^{2g}\,$, one gets
\[ I(\mathcal{G})\otimes_{\Z_\l} A[\l^\infty]^\vee \iri \L(\mathcal{G})\otimes_{\Z_\l}
A[\l^\infty]^\vee \sri A[\l^\infty]^\vee \ .\]
Since the mid term is projective (\cite[Lemma 4.2]{OV}), the previous sequence yields
\begin{equation}\label{SeqLon} H_1(\mathcal{H},A[\l^\infty]^\vee) \iri
(I(\mathcal{G})\otimes_{\Z_\l} A[\l^\infty]^\vee)_{\mathcal{H}} \to
(\L(\mathcal{G})\otimes_{\Z_\l} A[\l^\infty]^\vee)_{\mathcal{H}} \sri
(A[\l^\infty]^\vee)_{\mathcal{H}} \ .
\end{equation}
In order to shorten notations we put:
\begin{itemize}
\item[-] $X=H_1(\mathcal{H},A[\l^\infty]^\vee)\,$;
\item[-] $Y=(I(\mathcal{G})\otimes_{\Z_\l} A[\l^\infty]^\vee)_{\mathcal{H}}\,$;
\item[-] $J=Ker\{ (\L(\mathcal{G})\otimes_{\Z_\l} A[\l^\infty]^\vee)_{\mathcal{H}}
\sri (A[\l^\infty]^\vee)_{\mathcal{H}}\}\,$.
\end{itemize}
So the sequence \eqref{SeqLon} becomes
\begin{equation}\label{SeqShrt}
X \iri Y \sri J \ .
\end{equation}
For our purpose it is useful to think of $X$ as $H^1(F_S/K, A[\l^\infty])^\vee$ (note that
$H_1(\mathcal{H},A[\l^\infty]^\vee)\simeq H^1(\Omega/K, A[\l^\infty])^\vee\simeq H^1(F_S/K, A[\l^\infty])^\vee$).\\
Let $\mathcal{F}(d)$ denote a free pro-$\l$-group of rank $d=\dim \mathcal{G}$ and denote by $\mathcal{N}$ (resp.
$\mathcal{R}$) the kernel of the natural map $\mathcal{F}(d) \ri \mathcal{G}$ (resp. $\mathcal{F}(d) \ri G$).
For any profinite group $H$, we denote by $H^{ab}(\l)$ the maximal pro-$\l$-quotient of the maximal abelian quotient
of $H$. With this notations the powerful diagram reads as
\begin{equation}\label{PowDiag} \xymatrix{ H^2(\H,A[\l^\infty])^\vee \ar@{^(->}[r] \ar@{=}[d] &
(H^1(\mathcal{N}^{ab}(\l),A[\l^\infty])^\H)^\vee \ar[r] \ar[d]_{\simeq} &
H^1(\mathcal{R},A[\l^\infty]) \ar@{->>}[r] \ar@{^(->}[d] & X \ar@{^(->}[d] \\
H^2(\H,A[\l^\infty])^\vee \ar@{^(->}[r] &
(\mathcal{N}^{ab}(\l)\otimes A[\l^\infty]^\vee)_\H \ar[r] &
\L(G)^{2gd} \ar@{->>}[r] \ar@{->>}[d] & Y \ar@{->>}[d] \\
\ & \ & J \ar@{=}[r] &\ J \ .} \end{equation}
Moreover, since $cd_\l(\mathcal{G})\leqslant 2$ (just use \cite[Theorem 8.3.17]{NSW} and work as in \cite[Lemma 4.4, (iv)]{OV}),
the module $\mathcal{N}^{ab}(\l)\otimes A[\l^\infty]^\vee$ is free over $\L(\G)$ (\cite[Lemma 4.2]{OV}), hence
$(\mathcal{N}^{ab}(\l)\otimes A[\l^\infty]^\vee)_\H$ is projective as a $\L(\G/\H)=\L(G)$-module.
Therefore, if $H^2(F_S/K, A[\l^\infty])=0$, the module $Y$ has projective dimension $\leqslant 1$.
Whenever this is true the definition of $J$ provides the isomorphisms
\begin{equation}\label{BasicIso}
\E^i(X)\simeq\E^{i+1}(J)\qquad{\rm and}\qquad E^i(J)\simeq\E^{i+1}((A[\l^\infty]^\vee)_{\mathcal{H}})\quad
\forall\,i\geqslant2\ , \end{equation}
which will be repeatedly used in our computations.

\noindent We shall need also a ``localized'' version of the sequence \eqref{SeqShrt}.
For every $v\in S$ and a $w\in\M_K$ dividing $v$, we define
\[ X_v=H^1(K_w,A[\l^\infty])^\vee \qquad{\rm and}\qquad Y_v=(I(\G_v)\otimes_{\Z_\l} A[\l^\infty]^\vee)_{\H_v} \]
(with $\G_v$ the decomposition groups of $v$ in $\G$ and $\H_v=\H\cap\G_v\,$).
The exact sequence
\begin{equation}\label{SeqShrtLoc}
X_v \iri Y_v \sri J_v
\end{equation}
fits into the localized version of diagram \eqref{PowDiag}. If $K_w$ is still a local field, then Tate local duality
(\cite[Theorem 7.2.6]{NSW}) yields
\[ H^2(K_w,A[\l^\infty])=  H^2(K_w,\dl{n} A[\l^n])\simeq \il{n} H^0(K_w,A^t[\l^n])^\vee = 0 \ .\]
If $K_w$ is not local, then $\l^\infty$ divides the degree of the extension $K_w/F_v$ and $H^2(K_w,A[\l^\infty])=0$
by \cite[Theorem 7.1.8 (i)]{NSW}. Therefore $Y_v$ always has projective dimension $\leqslant 1$ and
\begin{equation}\label{BasicIsoLoc}
\E^i(X_v)\simeq\E^{i+1}(J_v)\simeq\E^{i+2}((A[\l^\infty]^\vee)_{\mathcal{H}_v})\quad
\forall\,i\geqslant2\ . \end{equation}

\noindent We note that, since $\l\neq p$, the image of the local Kummer maps is always 0, hence
\[ X_v=H^1(K_w,A[\l^\infty])^\vee = (H^1(K_w,A[\l^\infty])/Im\,\k_w)^\vee \simeq
H^1(K_w,A)[\l^\infty]^\vee \ .\]
Then Definition \ref{DefSel} for $L=K$ can be written as
\[ Sel_A(K)_\l = Ker \left\{ \psi\,:\,X^\vee \longrightarrow \bigoplus_S \Coind_G^{G_v} X_v^\vee \right\} \]
and, dualizing, we get a map
\[ \psi^\vee\,:\, \bigoplus_S \Ind^G_{G_v} X_v \longrightarrow X \]
whose cokernel is exactly $\mathcal{S}:=Sel_A(K)_\l^\vee\,$.\\

The following result will be fundamental for our computations.
\begin{thm}[U. Jannsen]\label{Jann}
Let $G$ be an $\l$-adic Lie group without elements of order $\l$ and of dimension $d$.
Let $M$ be a $\L(G)$-module which is finitely generated as $\Z_\l$-module. Then $\E^i(M)$
is a finitely generated $\Z_\l$-module and, in particular,
\begin{itemize}
\item[{\bf 1.}] if $M$ is $\Z_\l$-free, then $E^i(M)=0$ for any $i\neq d$ and $E^d(M)$ is free;
\item[{\bf 2.}] if $M$ is finite, then $E^i(M)=0$ for any $i\neq d+1$ and $E^{d+1}(M)$ is finite.
\end{itemize}
\end{thm}

\begin{proof}
See \cite[Corollary 2.6]{J}.
\end{proof}

\begin{cor}\label{JannCor} With notations as above:
\begin{itemize}
\item[{\bf 1.}] if $H^2(F_S/K,A[\l^\infty])=0$, then, for $i\geqslant 2$,
\[ \E^i(X)\ is\ \left\{ \begin{array}{ll}
finite & if\ i=d-1 \\
free & if\ i=d-2\\
 0 & otherwise \end{array} \right. \ ;\]
\item[{\bf 2.}] $\E_v^i\E_v^{i-1}(X_v)=0$ for $i\geqslant 3$.
\end{itemize}
\end{cor}

\begin{proof}
{\bf 1.} The hypothesis yields the isomorphism $\E^i(X)\simeq E^{i+2}((A[\l^\infty]^\vee)_{\mathcal{H}})$.
Since
\[ (A[\l^\infty]^\vee)_{\H}\simeq (A[\l^\infty]^{\H})^\vee=A[\l^\infty](K)^\vee\simeq \Z_\l^r \oplus \Delta \]
(with $0\leqslant r \leqslant 2g$ and $\Delta$ a finite group) and $\E^i(\Z_\l^r \oplus \Delta) =
\E^i(\Z_\l^r ) \oplus \E^i(\Delta)$, the claim follows from Theorem \ref{Jann}.\\
{\bf 2.} Use Theorem \ref{Jann} and the isomorphism in \eqref{BasicIsoLoc}.
\end{proof}

\begin{lem}\label{OcL3}
If $H^2(F_S/K,A[\l^\infty])=0$, then there is the following commutative diagram
\[ \begin{xy}
(0,0)*+{\E^1(Y)}="v1";
(35,0)*+{\bigoplus_S\Ind_{G_v}^G\E_v^1(Y_v)}="v2";
(70,0)*+{Coker(g_1)}="v3";
(0,-15)*+{\E^1(X)}="v4";
(35,-15)*+{\bigoplus_S\Ind_{G_v}^G\E_v^1(X_v)}="v5";
(70,-15)*+{Coker(h_1)}="v6";
(0,-30)*+{\E^2(J)}="v7";
(35,-30)*+{\bigoplus_S\Ind_{G_v}^G\E_v^2(J_v)}="v8";
(70,-30)*+{Coker(\bar{g}_1)\,.}="v9";
{\ar@{->}^{\!\!\!\!\!\!\!\!\!\!\!\!\!\!\!g_1}"v1";"v2"};
{\ar@{->>}"v2";"v3"};
{\ar@{->}^{\!\!\!\!\!\!\!\!\!\!\!\!\!\!\!h_1}"v4";"v5"};
{\ar@{->>}"v5";"v6"};
{\ar@{->}^{\!\!\!\!\!\!\!\!\!\!\!\!\!\!\!\bar{g}_1}"v7";"v8"};
{\ar@{->>}"v8";"v9"};
{\ar@{->}"v1";"v4"};
{\ar@{->>}"v4";"v7"};
{\ar@{->}"v2";"v5"};
{\ar@{->>}"v5";"v8"};
{\ar@{->}"v3";"v6"};
{\ar@{->>}^{f}"v6";"v9"};
\end{xy} \]
\end{lem}

\begin{proof} The inclusions $\mathcal{G}_v\subseteq\mathcal{G}$ and $\mathcal{H}_v\subseteq\mathcal{H}$ induce
the maps
\[ (I(\mathcal{G}_v)\otimes_{\Z_\l}A[\l^\infty]^\vee)_{\mathcal{H}_v}
\to(I(\mathcal{G}) \otimes_{\Z_\l}A[\l^\infty]^\vee)_{\mathcal{H}_v}
\to (I(\mathcal{G}) \otimes_{\Z_\l}A[\l^\infty]^\vee)_{\mathcal{H}} \ .\]
We have a homomorphism of $\L(G)$-modules $g:\bigoplus_S\Ind_{G_v}^G Y_v\to Y$
which, restricted to the $X_v$'s, provides the map $h: \bigoplus_S\Ind_{G_v}^G X_v\to X$.
So we have the following situation
\begin{equation}\label{1} \begin{xy}
(30,0)*+{\bigoplus_S\Ind_{G_v}^GX_v}="v1";
(30,-15)*+{\bigoplus_S\Ind_{G_v}^GY_v}="v2";
(30,-30)*+{\bigoplus_S\Ind_{G_v}^GJ_v}="v3";
(0,0)*+{X}="v4";
(0,-15)*+{Y}="v5";
(0,-30)*+{J}="v6";
{\ar@{^{(}->}"v1";"v2"};
{\ar@{->>}"v2";"v3"};
{\ar@{^{(}->}"v4";"v5"};
{\ar@{->>}"v5";"v6"};
{\ar@{->}_{\!\!\!\!h}"v1";"v4"};
{\ar@{->}_{\!\!\!\!g}"v2";"v5"};
{\ar@{-->}_{\!\!\!\!\bar{g}}"v3";"v6"};
\end{xy} \end{equation}
where $\bar{g}$ is induced by $g$ and the diagram is obviously commutative.\\
Since $Y$ and the $Y_v$'s have projective dimension $\leqslant 1$ (i.e., $\E^2(Y)=\E^2(Y_v)=0$),
the lemma follows by taking ${\rm Ext}$ in diagram \eqref{1} and recalling that, for any $i\geqslant 0$,
$\E^i_v(\Ind^G_{G_v}(X_v))=\Ind^G_{G_v}\E^i_v(X_v)$ (see \cite[Lemma 5.5]{OV}).
\end{proof}

In the next subsection we are going to describe the structure of $Coker(g_1)\,$.

\subsection{Homotopy theory and $Coker(g_1)$} For every finitely generated $\L(G)$-module $M$ choose a presentation
$ P_1 \to P_0 \to M \to 0 $ of $M$ by projectives
and define the {\em transpose} functor $DM$ by the exactness of the sequence
\[ 0 \to  \E^0(M) \to \E^0(P_0) \to \E^0(P_1) \to  DM \to 0.\]
Then it can be shown that the functor $D$ is well-defined and one has $D^2 = Id$
(see \cite{J}).



\begin{defin}\label{DefZmod}
{\em Let $L$ be an extension of $F$ contained in $F_S$. Then we define}
\[ Z(L):= H^0(F_S/L,\dl m D_2(A[\l^m]))^\vee\]
{\em where}
\[ D_2(A[\l^m])=\dl {F\subset E \subset F_S} (H^2(F_S/E,A[\l^m]))^\vee\]
{\em and the limit in $\displaystyle{\dl m D_2(A[\l^m])}$ is taken with respect to the $\l$-power
map $A[\l^{m+1}]\stackrel{\l}{\to}A[\l^m]$.}\\
{\em In the same way we define} $Z(L)$ {\em for any  Galois extension $L$ of $F_v$.}
\end{defin}

An alternative description of the module $Z$ is provided by the following

\begin{lem}\label{DescrZ}
Let $K$ be a fixed extension of $F$ contained in $F_S$ and $K_w$ its completion
for some $w|v\in S$. Then
\[ Z(K) \simeq \il {F\subseteq L \subseteq K} H^2(F_S/L,T_\l(A)) \quad {\rm and}\quad
 Z(K_w) \simeq \il {F_v\subseteq L \subseteq K_w} H^2(L,T_\l(A)) \ .\]
\end{lem}

\begin{proof} {\bf Global case.} For any global field $L$, let
\[ \ts^i(F_S/L,A[\l^\infty]):=Ker \left\{ H^i(F_S/L,A[\l^\infty]) \ri \bigoplus_S H^i(L_w,A[\l^\infty])\right\}\ .\]
We have already seen that $H^2(L_w,A[\l^\infty])=0$, hence
$H^2(F_S/L,A[\l^\infty])= \ts^2(F_S/L,A[\l^\infty])$. Using the pairing of
\cite[Ch. I, Proposition 6.9]{Mi1}, we get
\begin{align*}
Z(K) & =H^0(F_S/K, \displaystyle{\dl m \dl {F\subseteq L \subseteq F_S} \ts^2(F_S/L, A[\l^m])^\vee})^\vee\\
 & =\displaystyle{H^0(F_S/K,\dl m \dl {F\subseteq L \subseteq F_S} \ts^0(F_S/L, A^t[\l^m]))^\vee}\\
 & =(\displaystyle{\dl m \dl {F\subseteq L \subseteq F_S} \ts^0(F_S/L, A^t[\l^m])^{\Gal(F_S/K)}})^\vee\\
 & =(\displaystyle{\dl m \dl {F\subseteq L \subseteq K} \ts^0(F_S/L, A^t[\l^m])})^\vee\\
 & =\displaystyle{\il m \il {F\subseteq L \subseteq K} (H^2(F_S/L, A[\l^m])^\vee)^\vee}\\
 & =\displaystyle{\il {F\subseteq L \subseteq K} H^2(F_S/L,T_\l(A))\ .}
\end{align*}
{\bf Local case.} The proof is similar (using Tate local duality).
\end{proof}

We recall that our group $G$ has no elements of order $\l$, hence $\L(G)$ is a domain. Moreover
for any open subgroup $U$ of $G$ we have that (see \cite[Lemma 2.3]{J})
\[ \E^i(U)\simeq \E^i(G)\quad \forall\,i\in\Z\]
is an isomorphism of $\L(U)$-modules. An $\l$-adic Lie group $G$ always contains an
open pro-$\l$ subgroup (\cite[Corollary 8.34]{DdSMS}), so, in order to use properly the
usual definitions of ``torsion submodule'' and ``rank'' for a finitely generated $\L(G)$-module, with no loss
of generality, we will assume that $G$ is pro-$\l$.\\

\begin{prop}\label{ExtTor}
Let $M$ be a finitely generated $\L(G)$-module. Then $\E^i(M)$ is a finitely ge\-ne\-ra\-ted torsion $\L(G)$-module
for any $i\geqslant 1$.
\end{prop}

\begin{proof} Take a finite presentation $P_1 \to P_0 \to M \to 0$
with finitely generated and projective $\L(G)$-modules $P_1$ and $P_0\,$,
and the consequent exact sequence
\begin{equation}\label{FinPres}
 0\to R_1 \to P_0 \to M \to 0 \end{equation}
for a suitable submodule $R_1$ of $P_1\,$. Since $M$ and $\Hom_{\L(G)}(M,\L(G))$ have the same $\L(G)$-rank,
computing ranks in the sequence coming from \eqref{FinPres}
\begin{align*} \Hom_{\L(G)} & (M,\L(G)) \iri \Hom_{\L(G)}(P_0,\L(G)) \to \Hom_{\L(G)}(R_1,\L(G)) \to \E^1(M) \to \\
 & \to 0 \to \E^1(R_1) \to \E^2(M) \to 0 \to \cdots \to 0 \to \E^{i-1}(R_1) \to \E^i(M) \to 0 \to \cdots
\end{align*}
one finds $\rank_{\L(G)}(\E^1(M))=0$ for any finitely generated $\L(G)$-module $M$.
Therefore $\E^1(R_1)$ is torsion, which yields $\E^2(M)\simeq \E^1(R_1)$ is torsion. Iterating
$\E^i(M)\simeq \E^{i-1}(R_1)$ is $\L(G)$-torsion $\forall$ $i\geqslant 2$.
\end{proof}

\begin{lem}\label{OcVe2Thm33}
Let $F_n$ be subfields of $K$ such that $\Gal(K/F)={\displaystyle \il n}\, \Gal(F_n/F)\,$.
Then
\[ H^2_{Iw}(K_w,T_{\l}(A)):={\displaystyle \il {n,m}}\, H^2(F_{v_n},A[\l^m]) \]
is a torsion $\L(G_v)$-module. If $H^2(F_S/K,A[\l^\infty])=0$, then
\[ H^2_{Iw}(K,T_{\l}(A)):={\displaystyle \il {n,m}}\, H^2(F_S/F_n,A[\l^m]) \]
is a $\L(G)$-torsion as well.
\end{lem}

\begin{proof}
The proofs are identical so we only show the second statement. From the spectral sequence
\[ \E_2^{p,q}=E^p(H^q(F_S/K,A[\l^\infty])^\vee)\ \Longrightarrow\ H^{p+q}_{Iw}(K,T_\l(A))\]
due to Jannsen (see \cite{J1}), we have a filtration for $H^2_{Iw}(K,T_\l(A))$
\begin{equation}\label{filtr} 0= H^2_3 \subseteq H^2_2 \subseteq H^2_1 \subseteq H^2_0=H^2_{Iw}(K,T_\l(A)) \ ,
\end{equation}
which provides the following sequences:
\begin{align*} \E^0(H^1(F_S/K,A[\l^\infty])^\vee) & \to \E^2(H^0(F_S/K,A[\l^\infty])^\vee) \to  H^2_1 \\
 & \to \E^1(H^1(F_S/K,A[\l^\infty])^\vee) \to \E^3(H^0(F_S/K,A[\l^\infty])^\vee) \end{align*}
and
\[ H^2_1 \iri H^2_{Iw}(K,T_\l(A)) \sri \E_{\infty}^{0,2} \ . \]
By hypothesis $\E_{\infty}^{0,2}\simeq \E_2^{0,2}=0$, so $H^2_1 \simeq H^2_{Iw}(K,T_\l(A))$.\\
Since $H^i(F_S/K,A[\l^\infty])^\vee$ is a finitely generated $\L(G)$-module for $i\in\{0,1\}$
(for $i=1$ just look at $X$ in diagram \eqref{PowDiag}),
Proposition \ref{ExtTor} yields that the groups $\E^2(H^0(F_S/K,A[\l^\infty])^\vee)$ and
$\E^1(H^1(F_S/K,A[\l^\infty])^\vee)$ are $\L(G)$-torsion. Hence $H^2_1$ is torsion as well.
\end{proof}

\begin{lem}\label{OcL4}
With notations and hypotheses as in Lemma \ref{OcL3}, $Coker(g_1)$ is
finitely ge\-ne\-ra\-ted as $\Z_\l$-module.
\end{lem}

\begin{proof}
Lemma \ref{DescrZ} yields $Z(K)=H^2_{Iw}(K,T_\l(A))$ so, using \cite[Proposition 4.10]{OV}, one has
$DH^2_{Iw}(K,T_\l(A))\simeq Y$. Therefore $\E^1(DH^2_{Iw}(K,T_\l(A)))\simeq \E^1(Y)$.
Since $H^2_{Iw}(K,T_\l(A))$ is a $\L(G)$-torsion module, \cite[Lemma 3.1]{OV} implies
$\E^1(DH^2_{Iw}(K,T_\l(A))\simeq H^2_{Iw}(K,T_\l(A))$, i.e.,
\[ H^2_{Iw}(K,T_\l(A))\simeq \E^1(Y)  \]
(the same holds for the ``local'' modules).
The map $g_1$ of Lemma \ref{OcL3} then reads as
\[ g_1: \il n H^2(F_S/F_n,T_\l(A)) \to \bigoplus_S \Ind^G_{G_v} \il n H^2(F_{v_n},T_\l(A)) \ . \]
The claim follows from the Poitou-Tate sequence (see \cite[8.6.10 p. 488]{NSW}), since
\[ Coker(g_1)\simeq \il {n,m} H^0(F_S/F_n,(A[\l^m])') \ . \]
\end{proof}

\section{The main Theorem}\label{SecMain}
We are now ready to prove the following

\begin{thm}\label{NoPsNThm}
Let $G=\Gal(K/F)$ be an $\l$-adic Lie group without elements of order $\l$ and of
positive dimension $d\geqslant 3$. If
$H^2(F_S/K, A[\l^\infty])=0$ and the map $\psi$ in the sequence
\begin{equation}\label{Seqpsi} Sel_A(K)_\l \iri H^1(F_S/K, A[\l^\infty])
\stackrel{\psi}{\longrightarrow} \bigoplus_S \Coind_G^{G_v} H^1(K_w, A)[\l^\infty]
\end{equation}
is surjective, then $\mathcal{S}:=Sel_A(K)_\l^\vee$ has no nontrivial pseudo-null submodule.
\end{thm}

\begin{proof}
We need to prove that
\[\E^i\E^i(\mathcal{S})=0\ \ \forall\,i\geqslant 2 \ ,\]
and we consider two cases.

\noindent\underline{\bf Case $i=2$.}
Let $\D:= \bar{g}_1(\E^2(J))$. Then
\[ Coker(\bar{g}_1)= \bigoplus_S \Ind^G_{G_v}\E^2(J_v)/\D \ . \]
Observe that $\D\simeq  \bar{g}_1 (\E^3(A[\l^\infty]^\vee_{\H}))$
is a finitely generated $\Z_\l$-module (it is zero if
$d\neq 3$ and free as $\Z_\l$-module if $d=3$), so $\E^1(\D)=0$. Even if the theorem is limited to $d\geqslant 3$ we
remark here that, for $d=2$, $\D$ is finite and, for $d=1$, $\D=0$: hence $\E^1(\D)=0$ in any case.\\
Moreover
\begin{align*}
\E^2(\bigoplus_S \Ind^G_{G_v}\E^2(J_v))& =\E^2(\bigoplus_S \Ind^G_{G_v}\E^3(A[\l^\infty]^\vee_{\H_v}))\\
 & = \bigoplus_S \Ind^G_{G_v}\E^2\E^3(A[\l^\infty]^\vee_{\H_v})=0\ ,
\end{align*}
so, taking $\Ext$ in the sequence,
\begin{equation}\label{Seqfin}
 \D \iri  \bigoplus_S \Ind^G_{G_v}\E^2(J_v) \sri \bigoplus_S \Ind^G_{G_v}\E^2(J_v)/\D \ ,
\end{equation}
one finds
\begin{equation*}
\E^1(\D) \to  \E^2(\bigoplus_S \Ind^G_{G_v}\E^2(J_v)/\D) \to \E^2(\bigoplus_S \Ind^G_{G_v}\E^2(J_v)) \ .
\end{equation*}
Therefore
\begin{equation}\label{SeqQuot} \E^2(\bigoplus_S \Ind^G_{G_v}\E^2(J_v)/\D)=0 \ .
\end{equation}
Recall the sequences
\begin{equation}\label{SeqA}
\bigoplus_S \Ind^G_{G_v}X_v \iri X \sri \mathcal{S} \end{equation}
\begin{equation}\label{SeqB}
Ker(f) \iri Coker(h_1) \sri Coker(\bar{g}_1) \end{equation}
provided (respectively) by the hypothesis on $\psi$ and by Lemma \ref{OcL3}.
Take $\Ext$ on \eqref{SeqA} to get
\[  \E^1(X) \stackrel{h_1}{\longrightarrow} \E^1(\bigoplus_S \Ind^G_{G_v}X_v) \to  \E^2(\mathcal{S}) \to \E^2(X) \ .\]
If $d\geqslant 5$, then $\E^2(X)\simeq\E^3(J)\simeq\E^4(A[\l^\infty]^\vee_{\H})=0$. When this is the case
$Coker(h_1) \simeq \E^2(\mathcal{S})$ and sequence \eqref{SeqB} becomes
\[ Ker(f) \iri \E^2(\mathcal{S}) \sri \bigoplus_S \Ind^G_{G_v}\E^2(J_v)/\D \ .\]
By Lemma \ref{OcL4}, $Ker(f)$ is a finitely generated $\Z_\l$-module. Taking $\Ext$, one has
\[ \E^2(\bigoplus_S \Ind^G_{G_v}\E^2(J_v)/\D) \to \E^2\E^2(\mathcal{S}) \to \E^2(Ker(f)) \ ,\]
where the first and third term are trivial, so $\E^2\E^2(\mathcal{S})=0$ as well.\\
We are left with $d=3,4$. We know that $\E^4(A[\l^\infty]^\vee_{\H})=\E^2(X)$ is free over $\Z_\l$ if $d=4$
or finite if $d=3$ (again we remark it is 0 if $d=1,2$). Anyway $\E^2\E^2(X)=0$ in all cases. From the sequence
\[ Coker(h_1) \iri \E^2(\mathcal{S}) \stackrel{\eta}{\longrightarrow} \E^2(X) \]
one writes
\begin{equation}\label{C} Coker(h_1) \iri \E^2(\mathcal{S}) \sri Im(\eta) \end{equation}
where $Im(\eta)$ is free over $\Z_\l$ if $d=4$ or finite if $d=3$.\\
Taking $\Ext$ in \eqref{SeqB} one has
\[ \E^2(Coker(\bar{g}_1)) \to \E^2(Coker(h_1)) \to \E^2(Ker(f)) \]
with the first (see equation \eqref{SeqQuot}) and third term equal to zero, so $\E^2(Coker(h_1))=0$.
This fact in sequence \eqref{C} implies
\[ 0=\E^2(Im(\eta)) \to \E^2\E^2(\mathcal{S}) \to \E^2(Coker(h_1))=0 \ ,\]
so $\E^2\E^2(\mathcal{S})=0$.

\noindent\underline{\bf Case $i\geqslant 3$.}
From sequence \eqref{SeqA} we get the following
\begin{equation}\label{SeqThm}
\E^{i+1}(A[\l^\infty]_\H^\vee)\simeq\E^{i-1}(X)\to \bigoplus_S\Ind_{G_v}^G\E_v^{i-1}(X_v)\to \E^i(\mathcal{S})\to
\E^i(X)\simeq\E^{i+2}(A[\l^\infty]_\H^\vee) \ .
\end{equation}
We have four cases, depending on whether $\E^{i-1}(X)$ and $\E^i(X)$ are trivial or not.

\noindent{\bf Case 1.} Assume $\E^{i-1}(X)=\E^i(X)=0$.\\
From \eqref{SeqThm} we obtain the isomorphism
\[ \bigoplus_S\Ind_{G_v}^G\E_v^{i-1}(X_v)\simeq \E^i(\mathcal{S})\ , \]
so
\[ \bigoplus_S\Ind_{G_v}^G\E_v^i\E_v^{i-1}(X_v)\simeq \E^i\E^i(\mathcal{S})=0\]
thanks to Corollary \ref{JannCor} part {\bf 2}. We remark that this is the only case to consider when $d=1,2$.

\noindent{\bf Case 2.} Assume $\E^{i-1}(X)=0$ and $\E^i(X)\neq 0$.\\
This happens when $i=d-2$ or $i=d-1$ and $A[\l^\infty]^\vee_{\mathcal{H}}$ is finite.
From \eqref{SeqThm} we have
\[\bigoplus_S\Ind_{G_v}^G \E_v^{d-3}\iri \E^{d-2}(\mathcal{S})\sri N\]
\[(\ {\mathrm{resp.}}\quad \bigoplus_S\Ind_{G_v}^G \E_v^{d-2}\iri \E^{d-1}(\mathcal{S})\sri N\ )\]
where $N$ is a submodule of the free module $\E^{d-2}(X)$ (resp. of the finite module $\E^{d-1}(X)$). Therefore
$\E^{d-2}(N)=0$ (resp. $\E^{d-1}(N)=0$) and, moreover, $\E_v^{d-2}\E_v^{d-3}(X_v)=0$ (resp. $\E_v^{d-1}\E_v^{d-2}(X_v)=0$)
by Corollary \ref{JannCor} part {\bf 2}. Hence
$\E^{d-2}\E^{d-2}(\mathcal{S})=0$ (resp. $\E^{d-1}\E^{d-1}(\mathcal{S})=0$).

\noindent{\bf Case 3.} Assume $\E^{i-1}(X)\neq0$ and $\E^i(X)=0$.\\
This happens when $i=d$ or $i=d-1$ and $A[\l^\infty]^\vee_{\mathcal{H}}$ is free. The sequence \eqref{SeqThm} gives
\[ N\to \bigoplus_S\Ind_{G_v}^G\E_v^{d-1}(X_v)\sri \E^d(\mathcal{S}) \]
\[ (\ {\mathrm{resp.}}\quad N\to \bigoplus_S\Ind_{G_v}^G\E_v^{d-2}(X_v)\sri \E^{d-1}(\mathcal{S})\ )\]
where now $N$ is a quotient of the finite module $\E^{d-1}(X)$ (resp. of the free module $\E^{d-2}(X)\,$). Then
$\E^d(N)=0$ (resp. $\E^{d-1}(N)=0$) and
\[ \bigoplus_S\Ind_{G_v}^G\E_v^d\E_v^{d-1}(X_v)\simeq \E^d\E^d(\mathcal{S})=0  \]
\[(\ {\mathrm{resp.}}\quad \bigoplus_S\Ind_{G_v}^G\E_v^{d-1}\E_v^{d-2}(X_v)\simeq \E^{d-1}\E^{d-1}(\mathcal{S})=0\ )\ . \]

\noindent{\bf Case 4.} Assume $\E^{i-1}(X)\neq0$ and $\E^i(X)\neq 0$.\\
This happens when $i=d-1$ and $A[\l^\infty]^\vee_{\mathcal{H}}$ has nontrivial rank and torsion.
From sequence \eqref{SeqThm} we have
\[ \E^{d-2}(X)\to\bigoplus_S\Ind_{G_v}^G\E_v^{d-2}(X_v)\to \E^{d-1}(\mathcal{S})\to \E^{d-1}(X)\ . \]
Let $N_1,N_2$ and $N_3$ be modules such that:
\begin{itemize}
\item[-] $N_1$ is a quotient of $\E^{d-2}(X)$ (which is torsion free so that $\E^{d-2}(N_1)=0$);
\item[-] $N_2$ is a submodule of $\E^{d-1}(X)$ (which is finite so that $\E^{d-1}(N_2)=0$);
\item[-] $N_3$ is a module such that the sequences
\[ N_1\iri \bigoplus_S\Ind_{G_v}^G\E_v^{d-2}(X_v) \sri N_3\qquad
\mathrm{and}\qquad N_3\iri \E^{d-1}(\mathcal{S})\sri N_2 \]
 are exact.
\end{itemize}
Applying the functor $\Ext$ we find
\[ \E^{d-2}(N_1)\to \E^{d-1}(N_3)\to \bigoplus_S\Ind_{G_v}^G\E_v^{d-1}\E_v^{d-2}(X_v) \]
(which yields $\E^{d-1}(N_3)=0$), and
\[ \E^{d-1}(N_2)\to\E^{d-1}\E^{d-1}(\mathcal{S})\to\E^{d-1}(N_3) \]
which proves $\E^{d-1}\E^{d-1}(\mathcal{S})=0$.
\end{proof}

\begin{rem}\label{d=1,2}
{\em As pointed out in various steps of the previous proof, most of the statements
still hold for $d=1,2$. The only missing part is $\E^2(Ker(f))=0$ for $i=2$, in that case only our
calculations to get $\E^2\E^2(\mathcal{S})=0$ fail. In particular the same proof shows that $\E^2\E^2(\mathcal{S})=0$
when $Ker(f)$ is free and $d=1$ or when $Ker(f)$ is finite and $d=2$ or, obviously, for any $d$ if $f$ is injective.}
\end{rem}

We can extend the previous result to the $d\geqslant 2$ case with some extra assumptions.

\begin{prop}\label{dimcm2}
Let $G=\Gal(K/F)$ be an $\l$-adic Lie group without elements of order $\l$ and of
dimension $d\geqslant 2$. If $H^2(F_S/K, A[\l^\infty])=0$ and
$\cd_\l(G_v)=2$ for any $v\in S$, then $Sel_A(K)_\l^\vee$ has no nontrivial pseudo-null submodule. 
\end{prop}

\begin{proof}
Since $\cd_\l(F_v)=2$ (by \cite[Theorem 7.1.8]{NSW}), our hypothesis implies that $\Gal(\ov{F_v}/K_w)$ has no
elements of order $l$ (see also \cite[Theorem 7.5.3]{NSW}). Hence $H^1(K_w,A[\l^\infty])^\vee=0$ and
$Sel_A(K)_\l^\vee \simeq X$ embeds in $Y$. Now $H^2(F_S/K,A[\l^\infty])=0$ yields $Y$ has projective dimension
$\leqslant 1$, so $Y$ has no nontrivial pseudo-null submodule (by \cite[Proposition 2.5]{OV}).
\end{proof}

\subsection{The hypotheses on $H^2(F_S/K, A[\l^\infty])$ and $\psi$}
Let $F_m$ be extensions of $F$ such that $\Gal(K/F)\simeq \displaystyle{\il m \Gal(F_m/F)}$.
To provide some cases in which the main hypotheses hold we consider the
Poitou-Tate sequence for the module $A[\l^n]$, from which one can extract the sequence
\begin{equation}\label{CasSel}   \begin{xy}
(-15,0)*+{0}="v1"; (5,0)*+{Ker(\psi_{m,n})}="v2"; (40,0)*+{H^1(F_S/F_m, A[\l^n])}="v3";
(87,0)*+{\prod_{\begin{subarray}{c} v_m|v \\ v\in S\end{subarray}} H^1(F_{v_m}, A[\l^n])}="v4";
(87,-15)*+ {Ker(\psi^t_{m,n}))^\vee}="v5";(47,-15)*+{H^2(F_S/F_m, A[\l^n])}="v6";
(3,-15)*+{\prod_{\begin{subarray}{c} v_m|v \\ v\in S \end{subarray}}H^2(F_{v_m}, A[\l^n])}="v7";
(3,-30)*+{H^0(F_S/F_m,A^t[\l^n])^\vee}="v8";(33,-30)*+{0}="v9";
{\ar@{->}"v1";"v2"};{\ar@{->}"v2";"v3"};{\ar@{->}^<<<<<{\psi_{m,n}}"v3";"v4"};
{\ar@{->}^{\phi_{m,n}}"v4";"v5"};{\ar@{->}"v5";"v6"};{\ar@{->}"v6";"v7"};
{\ar@{->}"v7";"v8"};{\ar@{->}"v8";"v9"};
\end{xy} \end{equation}
(where $\psi_{m,n}^t$ is the analogue of $\psi_{m,n}$ for the dual abelian variety $A^t\,$, i.e., their kernels
represent the Selmer groups over $F_m$ for the modules $A^t[\l^n]$ and $A[\l^n]$ respectively). Taking direct limits on $n$
and recalling that $H^2(F_{v_m}, A[\l^\infty])=0$, the sequence \eqref{CasSel} becomes
\begin{equation}\label{Cassels} \begin{xy}
(-15,0)*+{0}="v1"; (5,0)*+{Sel_A(F_m)_\l}="v2"; (40,0)*+{H^1(F_S/F_m, A[\l^\infty])}="v3";
(85,0)*+{\prod_{\begin{subarray}{c} v_m|v \\ v\in S \end{subarray}}H^1(F_{v_m}, A[\l^\infty])}="v4";
(85,-15)*+{(\displaystyle{\il n Ker(\psi^t_{m,n})})^\vee}="v5";(40,-15)*+{H^2(F_S/F_m, A[\l^\infty])}="v6";
(13,-15)*+{0}="v7";
{\ar@{->}"v1";"v2"};{\ar@{->}"v2";"v3"};{\ar@{->}^<<<<{\psi_m}"v3";"v4"};{\ar@{->}^{\phi_m}"v4";"v5"};
{\ar@{->}"v5";"v6"};{\ar@{->}"v6";"v7"};
\end{xy}\end{equation}
(for more details one can consult \cite[Chapter 1]{CS}). One way to prove that $H^2(F_S/K, A[\l^\infty])=0$
and $\psi$ is surjective is to show that $\displaystyle{(\il n Ker(\psi^t_{m,n}))^\vee=0}$ for any $m$.
We mention here two cases in which the hypothesis on the vanishing of
$H^2(F_S/K, A[\l^\infty])$ is verified. The following is basically \cite[Proposition 1.9]{CS}.

\begin{prop}\label{PropH2Zero} Let $F_m$ be as above and assume that
$|Sel_{A^t}(F_m)_\l|<\infty$ for any $m$, then
\[H^2(F_S/K, A[\l^\infty])=0 \ .\]
\end{prop}

\begin{proof}
From \cite[Chapter I Remark 3.6 ]{Mi1} we have the isomorphism
\[ A^t(F_{v_m})^*\simeq H^1(F_{v_m},A[\l^\infty])^\vee \ , \]
where $A^t(F_{v_m})^*\simeq \displaystyle {\il n A^t(F_{v_m})/\l^nA^t(F_{v_m})}\,$.\\
Taking inverse limits on $n$ in the exact sequence
\[ A^t(F_m)/\l^nA^t(F_m) \iri Ker(\psi^t_{m,n}) \sri \ts(A^t/F_m)[\l^n] \ ,\]
and noting that $|\ts(A^t/F_m)[\l^\infty]| < \infty$ yields $T_\l(\ts(A^t/F_m))=0$, we find
\[ A^t(F_m)^* \simeq \il n Ker(\psi^t_{m,n}) \ . \]
Therefore \eqref{Cassels} becomes
\begin{equation}\label{CasselsMod} \begin{xy}
(-15,0)*+{0}="v1"; (5,0)*+{Sel_A(F_m)_\l}="v2"; (40,0)*+{H^1(F_S/F_m, A[\l^\infty])}="v3";
(85,0)*+{\prod_{\begin{subarray}{c} v_m|v \\ v\in S \end{subarray}} (A^t(F_{v_m})^*)^\vee}="v4";
(85,-15)*+{(A^t(F_m)^*)^\vee}="v5";(47,-15)*+{H^2(F_S/F_m, A[\l^\infty])}="v6";
(17,-15)*+{0}="v7";
{\ar@{->}"v1";"v2"};{\ar@{->}"v2";"v3"};{\ar@{->}^<<<<<{\psi}"v3";"v4"};{\ar@{->}^{\widetilde{\phi}}"v4";"v5"};
{\ar@{->}"v5";"v6"};{\ar@{->}"v6";"v7"};
\end{xy}\end{equation}
By hypothesis $A^t(F_m)^*$ is finite, therefore $H^2(F_S/F_m, A[\l^\infty])$ is finite as well. From the cohomology
of the sequence
\[ A[\l] \iri A[\l^\infty] {\buildrel \l\over{-\!\!\!-\!\!\!\twoheadrightarrow}} A[\l^\infty] \]
(and the fact that $H^3(F_S/F_m, A[\l])=0$, because $cd_\l(\Gal(F_S/F_m))=2$), one finds
\[ H^2(F_S/F_m, A[\l^\infty]) {\buildrel \l\over{-\!\!\!-\!\!\!\twoheadrightarrow}} H^2(F_S/F_m, A[\l^\infty]) \ ,\]
i.e., $H^2(F_S/F_m, A[\l^\infty])$ is divisible. Being divisible and finite $H^2(F_S/F_m, A[\l^\infty])$ must be 0
for any $m$ and the claim follows.
\end{proof}

We can also prove the vanishing of $H^2(F_S/K, A[\l^\infty])$ for the extension $K=F(A[\l^\infty])$.

\begin{prop}\label{PropTriv}
If $K=F(A[\l^\infty])$, then $H^2(F_S/K, A[\l^\infty])=0$.
\end{prop}

\begin{proof}
$\Gal(F_S/K)$ has trivial action on $A[\l^\infty]$ and (by the Weil pairing) on $\boldsymbol{\mu}_{\l^\infty}$, so
\[ H^2(F_S/K, A[\l^\infty])\simeq H^2(F_S/K, (\Q_\l/\Z_\l)^{2g}) \simeq H^2(F_S/K, (\boldsymbol{\mu}_{\l^\infty})^{2g}) \ .\]
Let $F_n=F(A[\l^n])$, using the notations of Lemma \ref{DescrZ}, Poitou-Tate duality (\cite[Theorem 8.6.7]{NSW})
and the isomorphism $\ts^1(F_S/F_n, \Z/\l^m\Z)\simeq\Hom(C\l_S(F_n),\Z/\l^m\Z)$ (\cite[Lemma 8.6.3]{NSW}), one has
\begin{align*}
H^2(F_S/K, \boldsymbol{\mu}_{\l^\infty}) & \simeq  \ts^2(F_S/K, \boldsymbol{\mu}_{\l^\infty})
  \simeq \dl {n,m} \ts^2(F_S/F_n, \boldsymbol{\mu}_{\l^m})\\
 & \simeq \dl {n,m} \ts^1(F_S/F_n, \boldsymbol{\mu}'_{\l^m})^\vee
  \simeq \dl {n,m} \ts^1(F_S/F_n, \Z/\l^m\Z)^\vee\\
 & \simeq \dl {n,m} \Hom(C\l_S(F_n),\Z/\l^m\Z)^\vee
  \simeq \dl {n,m} C\l_S(F_n)/\l^m \\
 & \simeq \dl n C\l_S(F_n) \otimes_\Z \Q_\l/\Z_\l =0
\end{align*}
since $C\l_S(F_n)$ is finite.
\end{proof}

\begin{rem}
{\em The above proposition works in the same way for a general $\l$-adic Lie extensions, unramified outside $S$,
which contains the trivializing extension. }
\end{rem}

\begin{exe}\label{AbVarTrExt}
{\em Let $A$ be an abelian variety without complex multiplication: by Proposition \ref{PropTriv},
the extension $K=F(A[\l^\infty])$ realizes the hypothesis of Proposition \ref{dimcm2} when every bad 
reduction prime is of split multiplicative reduction (in order to have $\cd_\l(G_v)=2$) and $\l>2g+1$ 
(by \cite{ST} and the embedding $\Gal(K/F)\iri{\rm GL}_{2g}(\Z_\l)\,$). Therefore $Sel_{A}(K)_\l^\vee$ has no nontrivial 
pseudo-null submodule. When $A=\mathcal{E}$ is an elliptic curve (using Igusa's theorem, see, e.g., \cite{BLV})
one can prove that $\dim\Gal(K/F)=4$ and also the surjectivity of the map $\psi$ (which, in this case,
is not needed to prove the absence of pseudo-null submodules): more details can be found in \cite{S}.\\
The same problem over number fields cannot (in general) be addressed in the same way and one needs the surjectivity
of the map $\psi$. The topic is treated (for example) in \cite[Section 4.2]{Co}.}
\end{exe}

\vspace{.5truecm}

\noindent Andrea Bandini\\
Universit\`a degli Studi di Parma - Dipartimento di Matematica e Informatica\\
Parco Area delle Scienze, 53/A - 43124 Parma - Italy\\
e-mail: andrea.bandini@unipr.it\\

\noindent Maria Valentino\\
Universit\`a della Calabria - Dipartimento di Matematica e Informatica\\
via P. Bucci - Cubo 31B - 87036 Arcavacata di Rende (CS) - Italy\\
e-mail: valentino@mat.unical.it


\begin{thebibliography}{99}

\bibitem{BLV} {\sc A. Bandini, I. Longhi, S. Vigni} {\em Torsion points on elliptic curves over
function fields and a theorem of Igusa},
Expo. Math. {\bf 27}, no. 3 (2009), 175--209.

\bibitem{BV} {\sc A. Bandini, M. Valentino} {\em Control Theorems for $\ell$-adic Lie extensions of global function fields},
arXiv:1206.2767v2 [math.NT].

\bibitem{B} {\sc N. Bourbaki} {\em Commutative Algebra}, Paris: Hermann (1972).

\bibitem{Co} {\sc J. Coates} {\em Fragments of the $GL_2$ Iwasawa theory of elliptic curves without complex multiplication},
in {\em Arithmetic Theory of Elliptic Curves, (Cetraro, 1997)}, Ed. C. Viola, Lecture Notes in Mathematics 1716,
Springer (1999), 1--50.

\bibitem{CS} {\sc J. Coates, R. Sujatha} {\em Galois cohomology of elliptic curves. 2nd Edition}
Narosa (2010).

\bibitem{DdSMS} {\sc J.D. Dixon, M.P.F. du Sautoy, A. Mann, D. Segal} {\em Analytic pro-$p$ groups. 2nd Edition},
Cambridge Studies in Advanced Mathematics {\bf 61}, Cambridge Univ. Press  (1999).


\bibitem{J1} {\sc U. Jannsen} {\em A spectral sequence for Iwasawa adjoints},
unpublished notes available on http://www.uni-regensburg.de/Fakultaeten/nat\_Fak\_I/Jannsen/index.html

\bibitem{J} {\sc U. Jannsen} {\em Iwasawa modules up to isomorphism}, Advanced
Studies in Pure Mathematics 17 (1989), Algebraic Number Theory - in honor of K. Iwasawa,
171--207.

\bibitem{L} {\sc M. Lazard} {\em Groupes analytiques p-adiques},
Publ. Math. I.H.E.S. {\bf 26} (1965), 389--603.

\bibitem{Mi1} {\sc J.S. Milne} {\em Arithmetic Duality Theorems},
BookSurge, LLC, Second edition, 2006.

\bibitem{NSW} {\sc J. Neuchirch, A. Schmidt, K. Wingberg} {\em Cohomology of number fields - Second edition},
GTM 323, Springer-Verlag, (2008).

\bibitem{O} {\sc Y. Ochi} {\em A note on Selmer groups of abelian varieties over the trivializing extensions},
Proc. Amer. Math. Soc. {\bf 134} (2006), no. 1, 31--37 .

\bibitem{OV} {\sc Y. Ochi, O. Venjakob} {\em On the structure of Selmer groups over p -adic Lie extensions},
J. Algebraic Geom. 11 (2002), no. {\bf 3}, 547--580.


\bibitem{S} {\sc G. Sechi} {\em $GL_2$ Iwasawa Theory of Elliptic Curves over Global Function Fields},
PhD thesis, University of Cambridge, (2006).

\bibitem{Se2} {\sc J.P. Serre} {\em Sur la dimension cohomologique des groupes profinis},
Topology {\bf 3} (1965), 413--420.

\bibitem{ST} {\sc J.P. Serre, J. Tate} {\em Good reduction of abelian varieties},
Ann. of Math. {\bf 88} (1968), 492--517.

\bibitem{Vj} {\sc O. Venjakob} {\em On the structure theory of the Iwasawa algebra of a $p$-adic Lie group},
J. Eur. Math. Soc. (JEMS) {\bf 4}, no. 3 (2002), 271--311.

\end{thebibliography}
\end{document}